\documentclass[12pt,a4paper]{amsart}

\usepackage[utf8]{inputenc}
\usepackage[T1]{fontenc}
\usepackage[english]{babel}
\usepackage{cite}

\usepackage{indentfirst}
\usepackage{amssymb}
\usepackage{amsfonts}
\usepackage{amsmath}
\usepackage{amsthm}
\usepackage[top=2.5cm, bottom=2.5cm, left=2.5cm, right=2.5cm]{geometry}
\usepackage{amsopn}
\usepackage[colorlinks]{hyperref}
\usepackage{mathrsfs}
\usepackage{amsxtra}
\usepackage{color}
\usepackage{dsfont}

\newcommand{\al}{\alpha}

\newcommand{\si}{\sigma}
\newcommand{\om}{\omega}
\newcommand{\la}{\lambda}
\newcommand{\va}{\varphi}

\theoremstyle{plain}
\newtheorem{thm}{Theorem}

\newtheorem{lem}{Lemma}[section]
\newtheorem{prop}{Proposition}[section]
\newtheorem{cor}{Corollary}[section]

\theoremstyle{definition}

\newtheorem*{example*}{Example}

\newtheorem*{rem*}{Remark}
\newtheorem{rem}[thm]{Remark}

\newcommand{\ep}{\eta}

\newcommand{\ve}{\varepsilon}
\newcommand{\Om}{\Omega}
\newcommand{\Ga}{\Gamma}
\newcommand{\ze}{\zeta}
\newcommand{\ga}{\gamma}
\newcommand{\Sd}{\mathbb{S}^{d-1}}

\newcommand{\R}{\mathbb{R}}

\DeclareMathOperator{\supp}{supp}

\DeclareMathOperator{\B}{B}

\DeclareSymbolFont{bbsymbol}{U}{bbold}{m}{n}
\DeclareMathSymbol{\ind}{\mathbin}{bbsymbol}{'061}

\title[Logarithmic fractional Hardy inequalities]{Logarithmic fractional Hardy inequalities}

\author[M{.} Kijaczko and V{.} Sahu]{Micha\l{} Kijaczko and Vivek Sahu}


\address[ M.K.]{Faculty of Pure and Applied Mathematics\\ Wroc{\l}aw University 
	of Science and Technology\\
	Wybrze\.ze Wyspia\'nskiego 27,
	50-370 Wroc{\l}aw, Poland
}
\email{michal.kijaczko@pwr.edu.pl}

\address[ V.S.]{Theoretical Statistics and Mathematics Unit,
Indian Statistical Institute, Delhi Centre, S.J. Sansanwal Marg, New Delhi, Delhi 110016, India}
\email{vivek@isid.ac.in,  viiveksahu@gmail.com}

\begin{document}
\begin{abstract}
We establish sharp fractional Hardy inequalities in the critical regimes $sp=d$ and $sp=1$, with explicit logarithmic weight functions. In the whole-space setting, for $p>d\ge 1$, $sp=d$, we prove that
\begin{equation*}
 \int_{\mathbb R^d}  \int_{\mathbb R^d}
\frac{|u(x)-u(y)|^p}{|x-y|^{2d}}\,dx\,dy
\geq
\left(\frac{p-1}{p}\right)^p
\int_{B_R}|u(x)|^p W_{p,d,R}(x)\,dx,   
\end{equation*}
for $u\in C_c(B_R)$, where $W_{p,d,R}$ is a logarithmic-type weight and the constant is sharp. We also show that the corresponding inequality with the simpler weight $|x|^{-d}(\log(R/|x|))^{-p}$ cannot hold with any positive constant. Furthermore, we derive two-sided estimates and asymptotic formulae for $W_{p,d,R}$, including its limiting connection with the classical Hardy inequality. In the half-space $\mathbb{R}^{d}_{+}$, we establish an analogous sharp fractional Hardy inequality in the critical case $sp=1$, with an explicit logarithmic weight depending on the distance to the boundary. We also investigate the asymptotic behaviour of this weight, obtaining estimates that describe its behaviour near the boundary. These results elucidate the role of logarithmic weights in critical fractional Hardy inequalities in both whole-space and half-space settings.
\end{abstract}

\maketitle

\section{Introduction}

Functional inequalities provide a framework for establishing fundamental estimates involving the norms and energies of functions. Among them, Hardy inequalities are characterized by the presence of singular weights and their relation to the underlying energy of a function. Such inequalities have been studied extensively in both local and nonlocal settings. Their study becomes particularly delicate in the critical case, where the classical inequality may fail to provide a nontrivial estimate and suitable modifications of the weight are required.

\subsection{Classical Hardy inequalities} Let $d\geq 2$ and $1\le p<d$. The classical Hardy inequality states that
\begin{equation}\label{Classical Hardy inequality}
\int_{\mathbb{R}^d}|\nabla u(x)|^p\,dx \geq \left(\frac{d-p}{p}\right)^p
\int_{\mathbb{R}^d}\frac{|u(x)|^p}{|x|^p}\,dx,
\end{equation}
for every $u\in C_c^1(\mathbb{R}^d)$. The constant $\left(\frac{d-p}{p}\right)^p$ is optimal. This inequality gives control over the weighted $L^p$-norm of a function by its gradient.

The situation changes when $p=d$. In this case, the constant in the classical Hardy inequality becomes zero, and hence the inequality does not provide any useful estimate. Moreover, the weight $|x|^{-d}$ is singular near the origin and is not locally integrable there. Therefore, a logarithmic correction is required to obtain a Hardy inequality in the critical case.

The inequality analogous to \eqref{Classical Hardy inequality} in the case $p=d=2$ was studied by Leray in \cite{leray1933} and was subsequently
extended to the case $p=d\geq 2$ in \cite{adi2002, tertikas2003}. In particular, by  \cite[Theorem $1.1$]{adi2002}, the following logarithmic Hardy inequality holds. Let $\Omega\subset\mathbb{R}^d$, where $d\geq 2$, be a bounded domain containing the origin. If $R>\sup_{x\in\Omega}\left(|x|e^{\frac{2}{d}}\right)$, then there exists a constant $C=C(d,\Omega,R)>0$ such that, for every $u\in W_0^{1,d}(\Omega)$,
\begin{equation*}
    \int_{\Omega}|\nabla u(x)|^d\,dx
    \geq C\int_{\Omega}
    \frac{|u(x)|^d}{|x|^d \left(\log\frac{R}{|x|}\right)^{\gamma}}
    \,dx,
\end{equation*}
if and only if $\gamma\geq d$.


Although this inequality provides a meaningful estimate in the critical case, it does not possess the same scale-invariance property as the classical Hardy inequality. To overcome this difficulty, Ioku and Ishiwata \cite{Ioku2015} established a scale-invariant form of
the critical Hardy inequality. Their main result states that
\begin{equation}\label{Critical Hardy inequality}
\int_{B_R}
\left|\frac{x}{|x|}\cdot\nabla u(x)\right|^d\,dx \geq \left(\frac{d-1}{d}\right)^d
\int_{B_R}
\frac{|u(x)|^d}
{|x|^d\left(\log\frac{R}{|x|}\right)^d}\,dx,
\end{equation}
for every $u\in W_0^{1,d}(B_R)$. Here, $B_R=\{x\in\mathbb{R}^d:|x|<R\}$ denotes the open ball in $\mathbb{R}^d$. The constant $\left(\frac{d-1}{d}\right)^d$ is optimal. Unlike the preceding inequality, the left-hand side involves only the radial component of the gradient. Moreover, this inequality is invariant under the scaling
\begin{equation*}
u_\lambda(x):=
\lambda^{-\frac{d-1}{d}}
u\left(R^{1-\la}|x|^{\lambda-1}x\right),
\quad \text{for } \lambda>0.
\end{equation*}

For further results on critical Hardy inequalities and logarithmic Hardy inequalities, we refer to some of the works in \cite{MR4168185, MR1724289, MR3101927, MR2671121}. For the special case $p=d=1$, we refer to \cite[Theorem~1]{Adi2025}, where the weight functions are expressed in terms of infinite logarithmic series; see \cite[Section~3]{Adi2025} for the one-dimensional proof.

\subsection{Fractional Hardy inequalities} A natural extension of the classical Hardy inequality is the fractional Hardy inequality. Frank and Seiringer \cite{MR2469027} established a sharp fractional Hardy inequality using a nonlinear ground state representation. They proved that, for $0<s<1$ and $1 \leq p < \frac{d}{s}$, the inequality holds for all $u  \in C_c(\mathbb{R}^d)$, while for $p > \frac{d}{s}$ it remains valid for $u \in C_c\left(\mathbb{R}^d \setminus \{0\}\right)$, and in both cases one has
\begin{equation}\label{Fractional Hardy}
\int_{\mathbb{R}^{d}}\int_{\mathbb{R}^{d}} \frac{|u(x)-u(y)|^{p}}{|x-y|^{d+sp}} \, dx \, dy \geq \mathcal{C}_{d,s,p} \int_{\mathbb{R}^{d}} \frac{|u(x)|^{p}}{|x|^{sp}} \, dx,
\end{equation}
where the constant $\mathcal{C}_{d,s,p}>0$ is optimal and admits the explicit representation
\begin{equation*}
\mathcal{C}_{d,s,p} := 2 \int_{0}^{1} r^{sp-1} \bigl|1-r^{(d-sp)/p}\bigr|^{p} \Phi_{d,s,p}(r) \, dr,
\end{equation*}
with
\begin{align*}
\Phi_{d,s,p}(r) &:= |\mathbb{S}^{d-2}| \int_{-1}^{1} \frac{(1-t^{2})^{(d-3)/2}}{(1-2rt+r^{2})^{(d+sp)/2}} \, dt, \quad d \geq 2, \\
\Phi_{1,s,p}(r) &:= \frac{1}{(1-r)^{1+sp}} + \frac{1}{(1+r)^{1+sp}}, \quad d=1.
\end{align*}
Since then, several generalisations and extensions of the above fractional Hardy inequalities in the non-critical case have been established. We refer the reader to \cite{MR4708667, MR4705882, VivekCCM2026} for some relevant results.

\smallskip

In the critical case $sp=d$, the weight $|x|^{-d}$ is singular both near the origin and at infinity. Consequently, as in the classical Hardy inequality in the critical case $p=d$, the above fractional Hardy inequality cannot be expected to hold with a positive constant, for all admissible functions. Therefore, an appropriate modification of the weight function is necessary to establish a suitable fractional Hardy inequality.

Nguyen and Squassina \cite[Theorem $3.1$]{MR3771839} considered the critical case $sp=d$ and established fractional Hardy inequalities with logarithmic improvements. In particular, they proved that, for $p>d\geq 1$ and $u\in C_c(\mathbb{R}^d)$ supported in $B_R$, the following inequality holds,
\begin{equation}\label{Critical fractional Hardy}
\int_{\mathbb{R}^d}\int_{\mathbb{R}^d}
\frac{|u(x)-u(y)|^p}{|x-y|^{2d}}
\,dx\,dy
 \geq C 
\int_{B_R}
\frac{|u(x)|^p}{|x|^d\left(\log\frac{2R}{|x|}\right)^p}
\,dx,
\end{equation}
where $C>0$ is a constant independent of $u$ and $R$. However, the optimality of the constant $C$ in \eqref{Critical fractional Hardy} remains an open question.

In fact, Nguyen and Squassina  \cite{MR3771839} established a full range of fractional Caffarelli--Kohn--Nirenberg inequalities in the critical case. By choosing appropriate parameters in their general inequality, we recover the above fractional Hardy inequality with a logarithmic improvement in the critical case $sp=d$.

\smallskip

The above fractional Hardy inequality in the critical case $sp=d$ raise the following questions:
\begin{itemize}
    \item[(1)]  \textit{Can the logarithmic weight functions in the above inequalities be improved? More generally, can one identify other weight functions that yield analogous fractional Hardy inequalities in the critical case $sp=d$ and recover the fractional Hardy inequality of the type \eqref{Critical fractional Hardy}?}

    \item[(2)] \textit{What is the optimal constant in such fractional Hardy inequalities in the critical case $sp=d$?}
\end{itemize}

In this article, we address these questions and establish fractional Hardy inequalities in the critical case with improved weight functions. Furthermore, we determine the sharp constant in the resulting inequalities, thereby providing an optimal version of the fractional Hardy inequality in the critical case $sp=d$. Moreover, the resulting inequality recovers the fractional Hardy inequality \eqref{Critical fractional Hardy} with an explicit constant. Below and elsewhere, $\om_{d-1}=\frac{2\pi^{d/2}}{\Ga\left(d/2\right)}$ denotes the surface measure of the unit sphere $\Sd\subset\R^d$. 

The following theorem presents the main result of this article.

\begin{thm}\label{Thm: logarithmic frac Hardy inequality}
Let $p>d\geq 1$, and let
$u\in C_c(B_R)$, where $B_{R}$ is a ball of radius $R>0$ in $\mathbb{R}^{d}$. Define
\begin{equation}\label{WpdR}
W_{p,d,R}(x):=\frac{2\om_{d-1}}{|x|^d}\int_{0}^{1}\frac{(1+t^2)t^{d-1}}{(1-t^2)^{d+1}}\frac{\left(\log\frac{1}{t}\right)^p}{\left(\log\frac{R}{t|x|}\right)^p}\,dt,\quad \text{for } x\in B_R\setminus\{0\}.
\end{equation}
Then the following fractional Hardy inequality holds,
\begin{equation}\label{Hardy with weight W}
\int_{\R^d}\int_{\R^d}\frac{|u(x)-u(y)|^p}{|x-y|^{2d}}\,dy\,dx\ge\left(\frac{p-1}{p}\right)^p\int_{B_R}|u(x)|^p W_{p,d,R}(x)\,dx.
\end{equation}
Moreover, the constant $((p-1)/p)^p$ in \eqref{Hardy with weight W} is sharp.
\end{thm}

Somewhat unexpectedly, the sharp constant in \eqref{Hardy with weight W} has the same form as the sharp constant in local Hardy inequalities. Although the weight $W_{p,d,R}$ is defined in a quite complicated way, we are able to describe its asymptotics near the origin and the boundary of $B_R$. More precisely, we show in Section \ref{Section 2}, that
\begin{align*}
 W_{p,d,R}(x)&\sim \frac{2\om_{d-1}C_{p,d}}{|x|^d\left(\log\frac{R}{|x|}\right)^p},\text{ as }|x|\rightarrow 0^+,\\
W_{p,d,R}(x)&\sim\frac{2^{-d+1}\om_{d-1}\B(d,p-d)}{R^d\left(\log\frac{R}{|x|}\right)^d},\quad\text{ as }|x|\rightarrow R^{-},
\end{align*}
where $C_{p,d}$ is an explicit constant given by \eqref{Cpd} and $\B$ is the Euler Beta function. In particular, the Hardy weight $W_{p,d,R}$ is singular both near the origin and the boundary of $B_R$, in contrast to the logarithmic weight appearing in \eqref{Critical fractional Hardy}.

 To obtain a more explicit lower bound for this weight, we introduce an additional parameter $a>1$. Using the definition of $W_{p,d,R}$ in \eqref{WpdR}, we observe that, for $\frac{1}{a}<t<1$,
\begin{equation*}
    \left(\log\frac{R}{t|x|}\right)^{-p}
    \geq
    \left(\log\frac{Ra}{|x|}\right)^{-p},
\end{equation*}
while for $0<t<\frac{1}{a}$, using $\log\frac{Ra}{|x|}>\log a$ for $0<|x|<R$,
$$
\log\frac{R}{t|x|}=\log\frac{Ra}{|x|}+\log\frac{1}{at}\le\log\frac{Ra}{|x|}\left(1+\frac{\log\frac{1}{at}}{\log a}\right)=\left( \log\frac{Ra}{|x|} \right) \frac{\log\frac{1}{t}}{\log a}
$$
Restricting the integral defining $W_{p,d,R}$ to the intervals $(0,1/a)$ and $(1/a,1)$, we obtain the following proposition.

\begin{prop}
Let $a>1$. Then,
\begin{equation}\label{estimates for W}
W_{p,d,R}(x)\geq
\frac{C_{d,p,a}}
{|x|^d\left(\log\frac{Ra}{|x|}\right)^p},
\end{equation}
where
\begin{equation*}
C_{d,p,a}:=2\omega_{d-1}\int_{0}^{1}
\frac{t^{d-1}(1+t^2)\left(\min\left\{\log a,\log\frac{1}{t}\right\}\right)^p}
{(1-t^2)^{d+1}}\,dt.
\end{equation*}
\end{prop}

The following corollary shows that the fractional Hardy inequality established in Theorem \ref{Thm: logarithmic frac Hardy inequality} recovers a fractional Hardy inequality of the form \eqref{Critical fractional Hardy}, with an explicit logarithmic weight and a constant inherited from the theorem and the above proposition.

\begin{cor}
Let $p>d\geq 1$, and let
$u\in C_c(B_R)$, where $R>0$. Then for any $a>1$, the following fractional Hardy inequality holds,
\begin{equation}\label{log Ra}
    \int_{\mathbb{R}^d}\int_{\mathbb{R}^d}
\frac{|u(x)-u(y)|^p}{|x-y|^{2d}}\,dy\,dx
\geq
\left(\frac{p-1}{p}\right)^p C_{d,p,a}
\int_{B_R}
\frac{|u(x)|^p}
{|x|^d\left(\log\frac{Ra}{|x|}\right)^p}\,dx.
\end{equation}
\end{cor}

It is natural to consider the limiting case $a\to 1^+$ in the preceding corollary. In this limit, the logarithmic weight $\log\frac{Ra}{|x|}$ approaches $\log\frac{R}{|x|}$. However, the resulting inequality does not hold with any positive constant, as the following theorem shows.
\begin{thm}\label{Thm: Failure of logarithmic Hardy inequality}
Let $p>d, \, d \geq 1$. It holds
$$
\inf\left\{\frac{\displaystyle\int_{\R^d}\int_{\R^d}\frac{|u(x)-u(y)|^p}{|x-y|^{2d}}\,dy\,dx}{\displaystyle\int_{B_R}\frac{|u(x)|^p}{|x|^{d}\left(\log\frac{R}{|x|}\right)^p}\,dx}\colon u\in C_c(B_R),\,u\not\equiv  0\right\}=0.
$$
In other words, there exists no positive constant $C=C(d,p,R)$ such that the inequality 
$$
\int_{\R^d}\int_{\R^d}\frac{|u(x)-u(y)|^p}{|x-y|^{2d}}\,dy\,dx\ge C\int_{B_R}\frac{|u(x)|^p}{|x|^{d}\left(\log\frac{R}{|x|}\right)^p}dx
$$
holds for all $u\in C_c(B_R)$.
\end{thm}

In the local setting, the Hardy inequality in the critical case holds with $a=1$; see \eqref{Critical Hardy inequality}. Moreover, this inequality can be recovered by taking the limit $p\to d$ in the fractional Hardy inequality established in Theorem \ref{Thm: logarithmic frac Hardy inequality}. In this regard, the Bourgain--Brezis--Mironescu formula \cite{MR3586796} (also see \cite[Corollary $3.20$]{MR4480576}) plays a crucial role. Recall that, for $u \in C^{1}_{c}(B_{R})$,
\begin{equation}\label{BBM}
  \lim_{s\rightarrow 1}(1-s)\int_{\R^d}\int_{\R^d}\frac{|u(x)-u(y)|^p}{|x-y|^{d+sp}}\,dy\,dx=\frac{2\pi^{\frac{d-1}{2}}\Gamma\left(\tfrac{1+p}{2}\right)}{p\,\Gamma\left(\tfrac{d+p}{2}\right)}\int_{B_R}|\nabla u(x)|^p\,dx.  
\end{equation}

In our case, $p=p(s)=\tfrac{d}{s}$ depends on $s$, therefore, \eqref{BBM} cannot be directly applied. However, in the Appendix (see Theorem \ref{Thm: BBM type Formula}, Section \ref{Section 4}), we derive the modified version of the Bourgain--Brezis--Mironescu formula with varying parameters, i.e. we show that
\begin{equation}\label{modifiedBBM}
\lim_{p \to d^+}(p-d)
\int_{\R^d}\int_{\R^d}
\frac{|u(x)-u(y)|^{p}}{|x-y|^{2d}}\,dy\,dx =
\frac{2\pi^{\frac{d-1}{2}}
\Gamma\left(\frac{1+d}{2}\right)}
{\Gamma(d)}
\int_{B_R}|\nabla u(x)|^d\,dx,\, \, u\in C_c^{1}(B_R).
\end{equation}
Also, note that
\begin{equation*}
  \frac{\pi^{\frac{d-1}{2}}\Gamma(\tfrac{1+d}{2})}{\Gamma(d)}= 2^{-d}\om_{d-1}.  
\end{equation*}
Furthermore, in Lemma \ref{Lemma: limit Wpdr} with $p>d$, we establish the following asymptotic behaviour of the weight function,
\begin{equation*}
 \lim_{p \to d^+}(p-d) W_{p,d,R}(x)
=
\frac{2^{1-d}\omega_{d-1}}
{|x|^d\left(\log\frac{R}{|x|}\right)^d}.
\end{equation*}
This asymptotic formula describes the behaviour of the fractional weight as $p$ approaches the critical exponent $d$ and provides a connection between the fractional and local Hardy inequalities. We also have 
\begin{equation*}
    \left(\frac{p-1}{p}\right)^{p}\rightarrow \left(\frac{d-1}{d}\right)^d,\quad \text{as } p\rightarrow d.
\end{equation*}

By taking the limit as $p\to d$ and using the above asymptotic behaviour of the weight function together with \eqref{modifiedBBM}, we obtain the following corollary, being a consequence of \eqref{Critical Hardy inequality}.
\begin{cor}
    Let $d \geq 2$ and $p=d$. Then for all $u \in C^{1}_{c}(B_{R})$, the following Hardy inequality holds,
\begin{equation*}
    \int_{B_R}|\nabla u(x)|^d\,dx\ge\left(\frac{d-1}{d}\right)^d\int_{B_R}\frac{|u(x)|^d}{|x|^d\left(\log\frac{R}{|x|}\right)^d}\,dx.
\end{equation*}
\end{cor}

\subsection{Fractional Hardy inequalities in half-space}

Let $\mathbb{R}^{d}_{+} :=\{x=(x',x_d)\in\mathbb{R}^{d-1} \times \mathbb{R}:x_d>0\}$. In the half-space, the distance to the boundary is given by \(x_d\), and the corresponding Hardy weight is
$x_d^{-sp}$. Bogdan and Dyda \cite{MR2663757} established the sharp fractional Hardy inequality in the half-space for $p=2$. This result was later extended by Frank and Seiringer \cite{MR2723817}, who employed a nonlinear and nonlocal ground state representation to establish the inequality for general $1 \leq p< \infty$ and $0<s<1$ in the
non-critical case $sp\neq 1$. They also identified the optimal
constant in the inequality. More precisely, for $ 1\leq p<\infty$, $0<s<1$, and $sp\neq 1$, they proved that
\begin{equation}\label{eq:FS fractional Hardy}
\int_{\mathbb{R}^{d}_{+}} \int_{\mathbb{R}^{d}_{+}}
\frac{|u(x)-u(y)|^p}{|x-y|^{d+sp}}\,dx\,dy \geq \mathcal{D}_{d,p,s} \int_{\mathbb{R}^{d}_{+}}
\frac{|u(x)|^p}{x_d^{sp}}\,dx,
\end{equation}
for all $u\in C_{c}(\mathbb{R}^{d}_{+})$, where the optimal
constant is given by
\begin{equation}\label{eq:FS sharp constant}
\mathcal{D}_{d,p,s} = 2\pi^{\frac{d-1}{2}} \frac{\Gamma\bigl(\frac{1+sp}{2}\bigr)} {\Gamma\bigl(\frac{d+sp}{2}\bigr)} \int_0^1 \frac{\bigl|1-r^{(sp-1)/p}\bigr|^p}
{(1-r)^{1+sp}}\,dr.
\end{equation}
The constant $\mathcal{D}_{d,p,s}$ is sharp, meaning that it cannot be replaced by a larger constant.

The case $sp=1$ is critical and is excluded from the above result. Indeed, the expression for the sharp constant in \eqref{eq:FS sharp constant} vanishes for $sp=1$. Thus, the critical case requires a separate analysis and cannot be obtained directly from the non-critical inequality above.

Fractional Hardy inequalities with optimal logarithmic corrections in the critical case $sp=1$ have been studied on bounded Lipschitz domains; see \cite{MR5026388, MR5018356}. Generalisations to smooth submanifolds can be found in \cite{MR5015195, MR4849884}. However, the optimality of the constants in these inequalities is not addressed in the aforementioned works.

\smallskip

 In the following theorem, we establish a fractional Hardy inequality in the half-space for $p>1$ with $sp=1$, under a suitable support condition on $u$. In particular, we obtain an explicit Hardy weight $V_{p,d,R}$ that depends on the distance to the boundary and the parameter $R$. For technical purposes, we fix a dimensional constant
$$
\kappa_d:=\frac{2\pi^{\frac{d-1}{2}}}
{\Gamma\left(\frac{d+1}{2}\right)}.
$$

\begin{thm}\label{thm:critical fractional Hardy half-space}
Let $p>1$ and $R>0$. Suppose that $u\in C_c(\mathbb{R}^d_+)$ satisfies $\operatorname{supp}u\subset A_R :=\{x\in\mathbb{R}^d_+:x_d<R\}$.  Then
\begin{equation}\label{halfspace}
\int_{\mathbb{R}^d_+}\int_{\mathbb{R}^d_+}
\frac{|u(x)-u(y)|^p}{|x-y|^{d+1}}\,dy\,dx
\geq
\left(\frac{p-1}{p}\right)^p
\int_{A_R}|u(x)|^p V_{p,d,R}(x_d)\,dx,
\end{equation}
where
\begin{equation}\label{VpdR}
V_{p,d,R}(x_d)
=
\frac{\kappa_d}{x_d}
\int_0^1
\frac{\left(\log\frac{1}{t}\right)^p}
{(1-t)^2\left(\log\frac{R}{tx_d}\right)^p}\,dt.
\end{equation} 

Moreover, the constant $((p-1)/p)^p$ in \eqref{halfspace} is sharp.
\end{thm}

In Section \ref{Section 3}, we derive the asymptotic formulae for the weight $V_{p,d,R}$ near the boundary of $A_R$. Namely, we show that 
\begin{align*}
V_{p,d,R}(x_d)&\sim \frac{\kappa_d\Ga(p+1)\ze(p)}{x_d\left(\log\frac{R}{x_d}\right)^p},\quad\text{ as }x_d\rightarrow 0^+,\\
V_{p,d,R}(x_d)&\sim\frac{\kappa_d}{(p-1)}\frac{1}{R\log\frac{R}{x_d}},\quad\text{ as }x_d\rightarrow R^-.
\end{align*}
In particular, $V_{p,d,R}$ is singular near the boundary of $A_R$. As in the case of the whole space, by a simple estimate, we have the following corollary.
\begin{cor}\label{corr:critical fractional Hardy half-space}
Let $p>1$, $R>0$, $a>1$. Suppose that $u\in C_c(\mathbb{R}^d_+)$ satisfies $\operatorname{supp}u\subset A_R=\{x\in\mathbb{R}^d_+:x_d<R\}$. Then, 
\begin{equation}\label{halfspace a>1}
\int_{\mathbb{R}^d_+}\int_{\mathbb{R}^d_+}
\frac{|u(x)-u(y)|^p}{|x-y|^{d+1}}\,dy\,dx
\geq
\left(\frac{p-1}{p}\right)^pD_{p,d,a}\int_{A_R}\frac{|u(x)|^p}{x_d\left(\log\frac{Ra}{x_d}\right)^p}\,dx,
\end{equation}
where 
$$
D_{p,d,a}=\kappa_d\int_{0}^1\frac{\left(\min\left\{\log a,\log\frac{1}{t}\right\}\right)^p}{(1-t)^2}\,dt.
$$
\end{cor}
By methods similar to the proof of Theorem \ref{Thm: Failure of logarithmic Hardy inequality}, one can show that the inequality \eqref{halfspace a>1} fails for $a=1$, with any positive constant. 

\smallskip

\noindent\textbf{Organisation of the paper.}
The paper is organised as follows. In Section \ref{Section 2}, we prove Theorems \ref{Thm: logarithmic frac Hardy inequality} and \ref{Thm: Failure of logarithmic Hardy inequality}. We first establish an auxiliary integral representation for the fractional energy of radial functions and use it to prove the critical fractional Hardy inequality and the sharpness of its constant. We then show that the corresponding inequality with the simpler logarithmic weight does not hold with any positive constant. In addition, we investigate the asymptotic behaviour of the weight function $W_{p,d,R}$ as $p$ approaches the critical exponent $d$, and establish two-sided estimates and asymptotic formulae for this weight. In Section \ref{Section 3}, we prove Theorem \ref{thm:critical fractional Hardy half-space}. We also provide two-sided estimates and asymptotic formulae for the weight function $V_{p,d,R}$. Finally, in Section \ref{Section 4} (Appendix), we prove a Bourgain--Brezis--Mironescu-type formula for varying exponents $p>d$, as $p\to d$.

\section{Proof of Theorems \ref{Thm: logarithmic frac Hardy inequality} and \ref{Thm: Failure of logarithmic Hardy inequality} and Asymptotic behaviour of \texorpdfstring{$W_{p,d,R}$}{WpdR}}\label{Section 2}


\smallskip

We begin with a lemma that expresses the fractional energy of a radial function in terms of its one-dimensional radial profile. This identity was established by Iula \cite[Proposition 2.1]{MR3729591} for $d=1$ and generalised by Parini and Ruf \cite[Proposition 4.3]{MR3797987} for $d\ge 2$. For the sake of completeness, we provide a simple and more elementary proof.

\begin{lem}\label{lem: critical Gagliardo seminorm}
Let $u\colon\R^d\to\R$ be a measurable radial function with a radial profile $v$. Then, 
\begin{equation}
\int_{\R^d}\int_{\R^d}\frac{|u(x)-u(y)|^p}{|x-y|^{2d}}\,dy\,dx=\om_{d-1}^2\int_{0}^{\infty}\int_{0}^{\infty}\frac{|v(r)-v(t)|^p}{|r^2-t^2|^{d+1}}(rt)^{d-1}(r^2+t^2)\,dr\,dt.    
\end{equation}
\end{lem}
\begin{proof}
Using polar coordinates, we have
\begin{align*}
 \int_{\R^d}\int_{\R^d}\frac{|u(x)-u(y)|^p}{|x-y|^{2d}}\,dy\,dx&=\int_{0}^{\infty}\int_{0}^{\infty}\int_{\Sd}\int_{\Sd}\frac{|v(r)-v(t)|^p}{|r\ep-t\si|^{2d}}(rt)^{d-1}d\ep\,d\si\,dr\,dt.   
\end{align*}
Assume $t<r$. By rotation invariance and \cite[(3.665.2)]{MR3307944}, \cite[(15.3.3)]{MR1225604},
\begin{align*}
\int_{\Sd}\int_{\Sd}\frac{d\ep\,d\si}{|r\ep-t\si|^{2d}}&=\om_{d-1}\int_{\Sd}\frac{d\ep}{\left(r^2-2rt\ep_d+t^2\right)^{d}}=\om_{d-1} r^{-2d}\int_{\Sd}\frac{d\ep}{\left(1-\tfrac{2t}{r}\ep_d+\left(\tfrac{t}{r}\right)^2\right)^{d}}\\
&=\om_{d-1}^2 r^{-2d}\,_{2}\text{F}_1\left(d,\tfrac{d+2}{2},\tfrac{d}{2};\left(\tfrac{t}{r}\right)^2\right)=\om_{d-1}^2r^{-2d}\frac{1+\left(\tfrac{t}{r}\right)^2}{\left|1-\left(\tfrac{t}{r}\right)^2\right|^{d+1}}\\
&=\om_{d-1}^2\frac{r^2+t^2}{\left|r^2-t^2\right|^{d+1}},
\end{align*}
where $_{2}\text{F}_1$ is the Gauss hypergeometric function. This completes the proof of the lemma.
\end{proof}

The following lemma establishes a one-dimensional Hardy inequality with a logarithmic weight, which is the key ingredient in proving Theorem \ref{Thm: logarithmic frac Hardy inequality}.

\begin{lem}\label{key lemma}
Let $p>1$, $R>0$, and $0<y<1$. Suppose that $\operatorname{supp} f\subset (0,Ry]$. Then
\begin{align*}
\int_{0}^{\infty}\frac{|f(x)-f(xy)|^p}{x}\,dx\ge\left(\frac{p-1}{p}\right)^p\left(\log\frac{1}{y}\right)^p\int_{0}^{R}\frac{|f(x)|^p}{x\left(\log\frac{R}{x}\right)^p}\,dx.    
\end{align*}
\end{lem}
\begin{proof}
Let $a=\log\tfrac{1}{y}$ and $h(x)=f(Re^{-x})$. Substituting $x\mapsto Re^{-t}$, we obtain
\begin{align*}
\int_{0}^{\infty}\frac{|f(x)-f(xy)|^p}{x}\,dx&=\int_{0}^{R}\frac{|f(x)-f(xy)|^p}{x}\,dx=\int_{0}^{\infty}\left|f\left(Re^{-t}\right)-f\left(Re^{-(t+a)}\right)\right|^p\,dt\\
&=\int_{0}^{\infty}|h(t+a)-h(t)|^p\,dt=\sum_{k=1}^{\infty}\int_{(k-1)a}^{ka}|h(t+a)-h(t)|^p\,dt\\
&=\int_{0}^{a}\sum_{k=1}^{\infty}|h(t+ka)-h(t+(k-1)a)|^p\,dt.
\end{align*}
Let $u_k(t)=h(t+ka)$. By assumption, we have $u_0(t)=h(t)=f(Re^{-t})=0$ for $0<t<a$. Hence, by the classical discrete Hardy inequality (see \cite{MR1544414}),
\begin{align*}
 \int_{0}^{a}\sum_{k=1}^{\infty}&|h(t+ka)-h(t+(k-1)a)|^p\,dt\\
 &=\int_{0}^{a}\sum_{k=1}^{\infty}|u_{k}(t)-u_{k-1}(t))|^p\,dt\ge\left(\frac{p-1}{p}\right)^p\int_{0}^{a}\sum_{k=1}^{\infty}\frac{|u_{k}(t)|^p}{k^p}\,dt\\
 &=\left(\frac{p-1}{p}\right)^p\int_{0}^{a}\sum_{k=1}^{\infty}\frac{|h(t+ka)|^p}{k^p}\,dt\ge\left(\frac{p-1}{p}\right)^p a^p\int_{0}^{a}\sum_{k=1}^{\infty}\frac{|h(t+ka)|^p}{(t+ka)^p}\,dt\\
 &=\left(\frac{p-1}{p}\right)^p a^p\int_{0}^{\infty}\frac{|h(t)|^p}{t^p}\,dt=\left(\frac{p-1}{p}\right)^p a^p\int_{0}^{R}\frac{|f(x)|^p}{x\left(\log\frac{R}{x}\right)^p}\,dx.
\end{align*}
The proof is complete.
\end{proof}

The following corollary is obtained by applying Lemma \ref{key lemma} to a function $f$ supported in $(0,R)$, with $R'=R/y$. It provides a logarithmic Hardy estimate that will be used in the proof of Theorem \ref{Thm: logarithmic frac Hardy inequality}.

\begin{cor}\label{corr}
 Let $p>1$, $R>0$, $0<y<1$. Suppose that $\operatorname{supp} f\subset (0,R]$. Then
\begin{equation}\label{key inequality}
\int_{0}^{\infty}\frac{|f(x)-f(xy)|^p}{x}\,dx\ge\left(\frac{p-1}{p}\right)^p\left(\log\frac{1}{y}\right)^p\int_{0}^{R}\frac{|f(x)|^p}{x\left(\log\frac{R}{xy}\right)^p}\,dx.    
\end{equation}   
\end{cor}

The following lemma shows that replacing a function with a suitable radial function does not increase its fractional energy, while preserving its weighted integral. This reduction allows us to assume radial symmetry in the proof of Theorem \ref{Thm: logarithmic frac Hardy inequality}. This fact is essentially well known, however, we include its proof for clarity of the work.

\begin{lem}\label{radial lemma}
 Let $p,d\ge 1$, $0<s<1$ and $u\in L^p_{loc}(\R^d)$. Define
 $$
 U(r)=\left(\frac{1}{\om_{d-1}}\int_{\Sd}|u(r\om)|^{p}\,d\om\right)^{\frac{1}{p}},\quad \text{for } r>0
 $$
 and $\widetilde{u}(x)=U(|x|)$. Then\\
 \begin{align*}
(i)&\int_{\R^d}\int_{\R^d}\frac{|u(x)-u(y)|^p}{|x-y|^{d+sp}}\,dy\,dx\ge\int_{\R^d}\int_{\R^d}\frac{|\widetilde{u}(x)-\widetilde{u}(y)|^p}{|x-y|^{d+sp}}\,dy\,dx;\\
(ii)&\int_{\R^d}|u(x)|^pV(x)\,dx=\int_{\R^d}|\widetilde{u}(x)|^pV(x)\,dx
 \end{align*}
 for any radial and nonnegative function $V$.
 \end{lem}
 \begin{proof}
  For $\om\in\Sd$, we define 
  $$
  \Phi_{d,s,p}(r,t):=\int_{\Sd}\frac{d\si}{|r\om-t\si|^{d+sp}}=\int_{\Sd}\frac{d\si}{\left(r^2-2rt\si_d+t^2\right)^\frac{d+sp}{2}}.
  $$
  Using polar coordinates and reverse triangle inequality, we have
  \begin{align*}
   &\int_{\R^d}\int_{\R^d}\frac{|u(x)-u(y)|^p}{|x-y|^{d+sp}}\,dy\,dx\\
   &=\int_{0}^{\infty}\int_{0}^{\infty}\int_{\Sd}\int_{\Sd}\frac{|u(r\om)-u(t\si)|^p}{|r\om-t\si|^{d+sp}}(rt)^{d-1}\,d\om\,d\si\,dr\,dt\\
   &\ge\int_{0}^{\infty}\int_{0}^{\infty}(rt)^{d-1}\,dr\,dt\\
   &\times\left|\left(\int_{\Sd}\int_{\Sd}\frac{|u(r\om)|^p}{|r\om-t\si|^{d+sp}}\,d\om\,d\si\right)^{\frac{1}{p}}-\left(\int_{\Sd}\int_{\Sd}\frac{|u(t\si)|^p}{|r\om-t\si|^{d+sp}}\,d\om\,d\si\right)^{\frac{1}{p}}\right|^p\\
   &=\om_{d-1}\int_{0}^{\infty}\int_{0}^{\infty}|U(r)-U(t)|^p\Phi_{d,s,p}(r,t)(rt)^{d-1}\,dr\,dt\\
   &=\int_{0}^{\infty}\int_{0}^{\infty}\int_{\Sd}\int_{\Sd}\frac{|U(r)-U(t)|^p}{|r\om-t\si|^{d+sp}}(rt)^{d-1}\,d\om\,d\si\,dr\,dt\\
   &=\int_{\R^d}\int_{\R^d}\frac{|\widetilde{u}(x)-\widetilde{u}(y)|^p}{|x-y|^{d+sp}}\,dy\,dx
  \end{align*}
  and the proof of part (i) is complete. To proof (ii), we calculate directly
  \begin{align*}
  \int_{\R^d}|u(x)|^p V(x)\,dx&=\int_{0}^{\infty}\int_{\Sd}|u(r\om)|^pV(r)r^{d-1}\,d\om\,dr=\om_{d-1}\int_{0}^{\infty}|U(r)|^pV(r)r^{d-1}\,dr\\
  &=\int_{\R^d}|\widetilde{u}(x)|^p V(x)\,dx.
  \end{align*}
 \end{proof}

 We are now ready to prove Theorem \ref{Thm: logarithmic frac Hardy inequality}. 

\begin{proof}[\textbf{Proof of Theorem \ref{Thm: logarithmic frac Hardy inequality}}]
By Lemma \ref{radial lemma}, we may, without loss of generality, assume that $u$ is radial. Applying Lemma \ref{lem: critical Gagliardo seminorm}, we have
\begin{align*}
\int_{\R^d}\int_{\R^d}\frac{|u(x)-u(y)|^p}{|x-y|^{2d}}\,dy\,dx&=\om_{d-1}^2\int_{0}^{\infty}\int_{0}^{\infty}\frac{|u(r)-u(t)|^p}{|r^2-t^2|^{d+1}}(rt)^{d-1}(r^2+t^2)\,dr\,dt\\
&=2\om_{d-1}^2\int_{0}^{\infty}\int_{0}^{r}\frac{|u(r)-u(t)|^p}{|r^2-t^2|^{d+1}}(rt)^{d-1}(r^2+t^2)\,dt\,dr\\
&=2\om_{d-1}^2\int_{0}^{1}\frac{t^{d-1}(1+t^2)}{(1-t^2)^{d+1}}dt\int_{0}^{\infty}\frac{|u(r)-u(rt)|^p}{r}\,dr.
\end{align*}
By Corollary \ref{corr},
\begin{align*}
\int_{0}^{\infty}\frac{|u(r)-u(rt)|^p}{r}\,dr&\ge\left(\frac{p-1}{p}\right)^p\left(\log\frac{1}{t}\right)^p\int_{0}^{R}\frac{|u(r)|^p}{r\left(\log\frac{R}{tr}\right)^p}\,dr\\
&=\frac{\left(\frac{p-1}{p}\right)^p\left(\log\frac{1}{t}\right)^p}{\om_{d-1}}\int_{B_R}\frac{|u(x)|^p}{|x|^d\left(\log\frac{R}{t|x|}\right)^p}\,dx.
\end{align*}
Combining the above estimates, we obtain the desired result.

We will now justify sharpness of the constant. To this end, we denote $\ga=\tfrac{p-1}{p}$, $c_p=\ga^p$ and 
$$
C_{p,d}=\int_{0}^{1}\frac{t^{d-1}\left(1+t^2\right)\left(\log\frac{1}{t}\right)^p}{\left(1-t^2\right)^{d+1}}\,dt.
$$
We define test functions
$$
u_N(r)=\begin{cases}
N^{\ga}-1,&\quad 0\le r\le R e^{-N},\\
\left(\log\frac{R}{r}\right)^{\ga}-1,&\quad R e^{-N}\le r\le \tfrac{R}{e},\\
0,&\quad r>\tfrac{R}{e},
\end{cases}
$$
or simply $u_N(r)=\left(\min\left\{\log\frac{R}{r},N\right\}^{\ga}-1\right)_+$, and let $u_N(x):=u_N(|x|)$. We have $\text{supp }u_N=\overline{B_{R/e}}\subset B_R$ and $0\le u_N\le N^{\ga}-1$.

We first observe that, by Hölder's inequality, for $r>0$, $0<t<1$, and an absolutely continuous function $u$,
\begin{align*}
|u(r)-u(rt)|^p&=\left|\int_{rt}^{r}u'(\xi)\,d\xi\right|^p\le\int_{rt}^{r}|u'(\xi)|^p\xi^{p-1}\,d\xi\left(\int_{rt}^{r}\frac{d\xi}{\xi}\right)^{p-1}\\
&=\left(\log\frac{1}{t}\right)^{p-1}\int_{rt}^{r}|u'(\xi)|^p\xi^{p-1}\,d\xi.
\end{align*}
Hence, by Fubini's theorem,
\begin{align*}
\int_{0}^{\infty}\frac{|u(r)-u(rt)|^p}{r}\,dr&\le \left(\log\frac{1}{t}\right)^{p-1}\int_{0}^{\infty}\frac{1}{r}\int_{rt}^{r}|u'(\xi)|^p\xi^{p-1}\,d\xi\,dr\\
&=\left(\log\frac{1}{t}\right)^{p}\int_{0}^{\infty}|u'(\xi)|^p\xi^{p-1}\,d\xi.
\end{align*}
Using the above estimates and Lemma \ref{lem: critical Gagliardo seminorm}, we have 
\begin{align*}
\int_{\R^d}\int_{\R^d}\frac{|u_N(x)-u_N(y)|^p}{|x-y|^{2d}}\,dy\,dx&=2\om_{d-1}^2\int_{0}^{1}\frac{t^{d-1}(1+t^2)}{(1-t^2)^{d+1}}\,dt\int_{0}^{\infty}\frac{|u_N(r)-u_N(rt)|^p}{r}\,dr\\
&\le2\om_{d-1}^2 C_{p,d}\int_{0}^{\infty}|u_N'(r)|^pr^{p-1}\,dr\\
&=2c_p\,\om_{d-1}^2 C_{p,d} \log N.
\end{align*}
We will now estimate the right-hand side. For any $1<M<N$, we have
\begin{align*}
 \int_{B_R}|u_N(x)|^p W_{p,d,R}(x)\,dx&\ge\om_{d-1}\int_{Re^{-N}}^{Re^{-M}}|u_{N}(r)|^p W_{p,d,R}(r)r^{d-1}\,dr\\
 &=2\om_{d-1}^2\int_{0}^{1}\frac{t^{d-1}(1+t^2)\left(\log\frac{1}{t}\right)^p}{(1-t^2)^{d+1}}\,dt\int_{Re^{-N}}^{Re^{-M}}\frac{\left(\left(\log\frac{R}{r}\right)^\ga-1\right)^p}{r\left(\log\frac{R}{r}+\log\frac{1}{t}\right)^p}\,dr.
\end{align*}
When $Re^{-N}\le r\le Re^{-M}$, we have $M\le\log\frac{R}{r}\le N$, so 
$$
\frac{\left(\left(\log\frac{R}{r}\right)^\ga-1\right)^p}{\left(\log\frac{R}{r}+\log\frac{1}{t}\right)^p}=\frac{1}{\log \frac{R}{r}}\frac{\left(1- \left(\log\frac{R}{r}\right)^{-\ga}\right)^p}{\left( 1+\frac{\log\frac{1}{t}}{\log\frac{R}{r}} \right)^{p}}\ge\frac{(1-M^{-\ga})^p}{\log\frac{R}{r}\left(1+\frac{\log\frac{1}{t}}{M}\right)^p}.
$$
Consequently,
\begin{align*}
& \int_{B_R}|u_N(x)|^p W_{p,d,R}(x)\,dx \\ & \qquad \ge2\om_{d-1}^2\left(1-M^{-\ga}\right)^p\int_{0}^{1}\frac{t^{d-1}(1+t^2)\left(\log\frac{1}{t}\right)^p}{(1-t^2)^{d+1}\left(1+\frac{\log\frac{1}{t}}{M}\right)^p}\,dt\int_{Re^{-N}}^{Re^{-M}}\frac{dr}{r\log\frac{R}{r}}\\
 & \quad =2\om_{d-1}^2\left(1-M^{-\ga}\right)^p\log\frac{N}{M}\int_{0}^{1}\frac{t^{d-1}(1+t^2)\left(\log\frac{1}{t}\right)^p}{(1-t^2)^{d+1}\left(1+\frac{\log\frac{1}{t}}{M}\right)^p}\,dt.
\end{align*}
Hence, the sharp constant $\Lambda$ in the inequality
$$
\int_{\R^d}\int_{\R^d}\frac{|u(x)-u(y)|^p}{|x-y|^{2d}}\,dy\,dx\ge\Lambda\int_{B_R}|u(x)|^pW_{p,d,R}(x)\,dx
$$
obeys
\begin{align*}
\Lambda&\le\frac{\displaystyle\int_{\R^d}\int_{\R^d}\frac{|u_N(x)-u_N(y)|^p}{|x-y|^{2d}}\,dy\,dx}{\displaystyle\int_{B_R}|u_N(x)|^pW_{p,d,R}(x)\,dx}\le\frac{c_pC_{p,d}\left(1-M^{-\ga}\right)^{-p}\log N}{\displaystyle\log\frac{N}{M}\int_{0}^{1}\frac{t^{d-1}(1+t^2)\left(\log\frac{1}{t}\right)^p}{(1-t^2)^{d+1}\left(1+\frac{\log\frac{1}{t}}{M}\right)^p}\,dt}.
\end{align*}
Letting now $N\rightarrow\infty$ and then $M\rightarrow\infty$, we obtain that $\Lambda\le c_p$. Together with the lower bound, it implies that $\Lambda=c_p$ and the proof is finished.
\end{proof}

We now prove Theorem \ref{Thm: Failure of logarithmic Hardy inequality} by constructing a suitable family of test functions and showing that the corresponding quotient tends to zero, thereby establishing the non-existence of a positive constant in the proposed inequality. In the proof below and elsewhere, we denote by 
$$
[u]_{W^{s,p}(\Om)}:=\left(\int_{\Om}\int_{\Om}\frac{|u(x)-u(y)|^p}{|x-y|^{d+sp}}\,dy\,dx\right)^{\frac{1}{p}
}
$$
the Gagliardo seminorm of $u$.

\begin{proof}[\textbf{Proof of Theorem \ref{Thm: Failure of logarithmic Hardy inequality}}]
Let $\va\in C_c^{1}\left(\left(1,2\right)\right)$ be a nonzero function satisfying $0\le\va\le 1$, and let $0<\ve<1$. We define $u_{\ve}(x)=\va\left(\frac{1}{\ve}\log\frac{R}{|x|}\right)$. It holds $\text{supp }u_{\ve}\subset\{R e^{-2\ve}\le|x|\le R e^{-\ve}\}\subset B_R$. Using polar integration and then substituting $r\mapsto R e^{-t}$, we have
\begin{align*}
\int_{B_R}\frac{|u_{\ve}(x)|^p}{|x|^{d}\left(\log\frac{R}{|x|}\right)^p}\,dx&=\om_{d-1}\int_{0}^{R}\frac{|\va\left(\frac{1}{\ve}\log\frac{R}{r}\right)|^p}{r\left(\log\frac{R}{r}\right)^p}dr=\om_{d-1}\int_{0}^{\infty}\frac{|\va\left(\tfrac{t}{\ve}\right)|^p}{t^p}\,dt\\
&=\om_{d-1}\ve^{1-p}\int_{1}^{2}\frac{|\va(t)|^p}{t^p}\,dt\approx \ve^{1-p}.
\end{align*}
Moreover, by Lemma \ref{lem: critical Gagliardo seminorm},
\begin{align*}
\int_{\R^d}\int_{\R^d}\frac{|u_{\ve}(x)-u_{\ve}(y)|^p}{|x-y|^{2d}}\,dy\,dx&=\om_{d-1}^2\int_{0}^{\infty}\int_{0}^{\infty}\frac{|u_{\ve}(r)-u_{\ve}(t)|^p}{|r^2-t^2|^{d+1}}(rt)^{d-1}(r^2+t^2)\,dr\,dt\\
&=\om_{d-1}^2\int_{\R}\int_{\R}\frac{\left|\va\left(\tfrac{z}{\ve}\right)-\va\left(\tfrac{w}{\ve}\right)\right|^p}{\left|e^{-2z}-e^{-2w}\right|^{d+1}}e^{-d(z+w)}\left(e^{-2z}+e^{-2w}\right)\,dz\,dw.
\end{align*}
Observe that, since it is elementary to check that the function $x\mapsto\cosh x\left(\frac{x}{\sinh x}\right)^{d+1}$ is bounded on $\R$, for some constant $C>0$ it holds 
\begin{align*}
\frac{e^{-d(z+w)}\left(e^{-2z}+e^{-2w}\right)}{\left|e^{-2z}-e^{-2w}\right|^{d+1}}=2^{-d}\frac{\cosh(z-w)}{|\sinh(z-w)|^{d+1}}\le \frac{C}{|z-w|^{d+1}},\,z\ne w.
\end{align*}
Consequently,
\begin{align*}
\int_{\R}\int_{\R}&\frac{\left|\va\left(\tfrac{z}{\ve}\right)-\va\left(\tfrac{w}{\ve}\right)\right|^p}{\left|e^{-2z}-e^{-2w}\right|^{d+1}}e^{-d(z+w)}\left(e^{-2z}+e^{-2w}\right)\,dz\,dw\le C\int_{\R}\int_{\R}\frac{\left|\va\left(\tfrac{z}{\ve}\right)-\va\left(\tfrac{w}{\ve}\right)\right|^p}{\left|z-w\right|^{d+1}}\,dz\,dw\\
&=C\ve^{1-d}[\va]_{W^{d/p,p}(\R)}^p\approx \ve^{1-d}.
\end{align*}
Summarising, 
\begin{align*} \frac{\displaystyle\int_{\R^d}\int_{\R^d}\frac{|u_{\ve}(x)-u_{\ve}(y)|^p}{|x-y|^{2d}}\,dy\,dx}{\displaystyle\int_{B_R}\frac{|u_{\ve}(x)|^p}{|x|^{d}\left(\log\frac{R}{|x|}\right)^p}\,dx}\approx \ve^{p-d}\rightarrow 0,
\end{align*}
when $\ve\rightarrow 0^+$, since $p>d$. The proof is complete.
\end{proof}

\subsection{Asymptotics of \texorpdfstring{$W_{p,d,R}$}{Wpdr}}\label{Subsecton: Asymptotic of WpdR} In this subsection, we establish the asymptotic behaviour of the weight function $W_{p,d,R}$ as $p$ approaches the critical exponent $d$. This analysis provides a connection between the fractional Hardy inequality and its classical counterpart.

\smallskip

We begin with an auxiliary lemma. Recall that the Hurwitz zeta function is defined for $\Re(z)>1$ and $a>0$ as
$$\ze(z,a)=\sum_{n=0}^{\infty}\frac{1}{(n+a)^{z}}.$$ 
It is well known that, for any $a>0$, this function has a simple pole at $z=1$ with residue being equal to one, i.e.
\begin{equation}\label{Hurwitz residue}
\lim_{z\rightarrow 1}(z-1)\ze(z,a)=1.
\end{equation}

\begin{lem}\label{lemma: limiting case p tends to d}
    For any $\lambda \geq 0$ and $d \geq 1$, we have
    \begin{equation*}
        \lim_{p \to d^+}  (p-d)     \sum_{k=0}^{\infty} \frac{(d+k-1)!}{k!} \frac{1}{(d+2k+ \lambda)^{p}} = 2^{-d}.
    \end{equation*}
\end{lem}
\begin{proof}

By a direct computation, 
\begin{equation}\label{Hurwitz}
\sum_{k=0}^{\infty}\frac{(d+k-1)!}{k!} \frac{1}{(d+2k+ \lambda)^{p}}=2^{1-d-p}\sum_{m=0}^{d-1}c_m2^m\ze\left(p-m,\tfrac{d+\la}{2}\right),
\end{equation}
where the coefficients $c_m=c_m(d,\la)$ satisfy
$$
\prod_{j=1}^{d-1}(x-d-\la+2j)=\sum_{m=0}^{d-1}c_mx^m.
$$
For $d=1$, we assume the product above to be equal to $1$. Since $c_{d-1}=1$, applying \eqref{Hurwitz residue} we have 
$$
\lim_{p\rightarrow d^+}(p-d)2^{1-d-p}\sum_{m=0}^{d-1}c_m2^m\ze\left(p-m,\tfrac{d+\la}{2}\right)=2^{-d}\lim_{p\rightarrow d^+}(p-d)\ze\left(p-d+1,\tfrac{d+\la}{2}\right)=2^{-d}
$$
and the proof is complete.

\end{proof}

Using the preceding lemma, we now establish the asymptotic behaviour of the weight function $W_{p,d,R}$ as $p\to d$, which will be used to derive the desired limiting formula.

\begin{lem}\label{Lemma: limit Wpdr}
    Let $p>d, \, d \geq 1$, and recall
    \begin{equation*}
        W_{p,d,R}(x)=\frac{2\om_{d-1}}{|x|^d}\int_{0}^{1}\frac{(1+t^2)t^{d-1}}{(1-t^2)^{d+1}}\frac{\left(\log\frac{1}{t}\right)^p}{\left(\log\frac{R}{t|x|}\right)^p}\,dt,\quad \text{for } x\in B_R\setminus\{0\}.
    \end{equation*}
    Then
    \begin{equation*}
          \lim_{p \to d^+} (p-d) W_{p,d,R}(x)  = \frac{2^{-d+1} \om_{d-1}}{ |x|^{d}\left( \log \frac{R}{|x|} \right)^{d}} .
    \end{equation*}
\end{lem}
\begin{proof}
Let $t=e^{-r}$. Then $dt = -e^{-r} \, dr$. Also,
\begin{equation*}
1+t^{2} = 2e^{-r} \cosh r, \quad 1-t^{2} = 2 e^{-r} \sinh r, \quad \log \frac{1}{t} = r, \quad t^{d-1} = e^{-r(d-1)}.
\end{equation*}
Therefore, using the change of variables as mentioned above, we have
\begin{equation}\label{eq2}
    W_{p,d,R}(x) = \frac{2^{-d+1}\om_{d-1}}{|x|^{d}} \int_{0}^{\infty} \frac{\cosh r}{(\sinh r)^{d+1}} \left( \frac{r}{r+ \log \frac{R}{|x|}} \right)^{p} \, dr = \frac{2^{-d+1}\om_{d-1}}{|x|^{d}} I ,
\end{equation}
where 
\begin{equation*}
    I := \int_{0}^{\infty} \frac{\cosh r}{(\sinh r)^{d+1}} \left( \frac{r}{r+ \log \frac{R}{|x|}} \right)^{p} \, dr.
\end{equation*}
Applying integration by parts and using $-\frac{1}{d} \left( \frac{\text{d}}{\text{d}r} (\sinh r)^{-d} \right) = \frac{\cosh r}{(\sinh r)^{d+1}}$, we obtain
\begin{align*}
    I & = - \frac{1}{d} \left[ \left( \frac{r}{r+ \log \frac{R}{|x|}} \right)^{p} \frac{1}{(\sinh r)^{d}}  \right]_{0}^{\infty} \\ & \quad + \frac{p}{d} \log \left( \frac{R}{|x|} \right) \int_{0}^{\infty} \frac{r^{p-1}}{(\sinh r)^{d} \left( r+ \log \frac{R}{|x|} \right)^{p+1}} \, dr. 
\end{align*}
As $r \to 0$, we have $\sinh r \sim r$. Therefore, using $p-d>0$, we have
\begin{equation*}
    \left( \frac{r}{r+ \log \frac{R}{|x|}} \right)^{p} \frac{1}{(\sinh r)^{d}} \to 0 \quad \text{as } r \to 0.
\end{equation*}
Also, $\sinh r$ diverges to $\infty$ as $r \to \infty$. Thus, we have
\begin{equation*}
    \left( \frac{r}{r+ \log \frac{R}{|x|}} \right)^{p} \frac{1}{(\sinh r)^{d}} \to 0 \quad \text{as } r \to \infty.
\end{equation*}
Hence, the boundary term vanishes. Consequently,
\begin{equation*}
    I  =  \frac{p}{d} \log \left( \frac{R}{|x|} \right) \int_{0}^{\infty} \frac{r^{p-1}}{(\sinh r)^{d} \left( r+ \log \frac{R}{|x|} \right)^{p+1}} \, dr.  
\end{equation*}
Using the definition of the $\Ga$ function, we have
\begin{equation*}
    \frac{1}{\left( r+ \log \frac{R}{|x|} \right)^{p+1}} = \frac{1}{\Gamma (p+1)} \int_{0}^{\infty} \lambda^{p} e^{-\lambda \left(r+ \log \frac{R}{|x|}\right)} \, d \lambda.
\end{equation*}
Also, using $0< e^{-2r} <1$ for all $r > 0$, and generalised binomial series,
\begin{align*}
    (\sinh r)^{-d} = 2^{d} e^{-dr} (1-e^{-2r})^{-d} = 2^{d} e^{-dr} \sum_{k =0}^{\infty} c_{k} e^{-2kr}= 2^{d} \sum_{k =0}^{\infty} c_{k} e^{-(d+2k)r},
    \end{align*}
where 
\begin{equation*}
    c_{k} := \frac{(d)_{k}}{k!} = \frac{\Gamma(d+k)}{\Gamma(d) \Gamma(k+1)}, \quad \text{for } k \geq 0.
\end{equation*}
Combining the above two relations to $I$, we have
\begin{equation*}
    I= \frac{2^{d}p}{d\,\Gamma(p+1)} \left( \log \frac{R}{|x|} \right) \int_{0}^{\infty} \lambda^{p}e^{-\lambda\left( \log \frac{R}{|x|}  \right)} \sum_{k =0}^{\infty} c_{k}  \left( \int_{0}^{\infty} r^{p-1} e^{-(d+2k+\lambda)r} \, dr \right) d \lambda.
\end{equation*}
Using the definition of the $\Gamma$ function together with $p\,\Gamma(p) = \Gamma(p+1)$, the above relation reduces to 
\begin{equation}\label{eq1}
  I=   \frac{2^{d}}{d} \left( \log \frac{R}{|x|} \right) \int_{0}^{\infty} \lambda^{p} e^{-\lambda \left( \log \frac{R}{|x|} \right)} \left( \sum_{k =0}^{\infty} \frac{c_{k}}{ (d+2k+ \lambda)^{p}}  \right) \, d \lambda.
\end{equation}
Note that
\begin{equation*}
    \sum_{k =0}^{\infty} \frac{c_{k}}{ (d+2k+ \lambda)^{p}} = \frac{1}{\Gamma(d)} \sum_{k =0}^{\infty} \frac{(d+k-1)!}{k!} \frac{1}{(d+2k+ \lambda)^{p}}.
\end{equation*}
From Lemma \ref{lemma: limiting case p tends to d}, for any $\lambda \geq 0$,
\begin{equation*}
 \lim_{p \to d^+}  (p-d)     \sum_{k =0}^{\infty} \frac{(d+k-1)!}{k!} \frac{1}{(d+2k+ \lambda)^{p}} = 2^{-d}.
\end{equation*}
Therefore, using \eqref{eq1} in the definition of $W_{p,d,R}$ (see \eqref{eq2}), taking limit $p \to d^+$, and using the dominated convergence theorem, we obtain
\begin{align*}
    \lim_{p \to d^+} (p-d) W_{p,d,R}(x) & = \frac{2^{-d+1} \om_{d-1}}{ |x|^{d}} \frac{1}{d \, \Gamma(d)}  \left( \log \frac{R}{|x|} \right) \int_{0}^{\infty} \lambda^{d} e^{-\lambda \left( \log \frac{R}{|x|} \right)} \, d \lambda \\ &  = \frac{2^{-d+1} \om_{d-1}}{ |x|^{d}\left( \log \frac{R}{|x|} \right)^{d}} .
\end{align*}
  This completes the proof of the lemma. 
\end{proof}

\subsection{Asymptotic behaviour of \texorpdfstring{$W_{p,d,R}$}{WpdR}}\label{Subsecton: Asymptotic behaviour WpdR} In this subsection, we establish two-sided bounds for the weight function $W_{p,d,R}$ and derive its asymptotic behaviour as $|x|\to 0^+$ and $|x|\to R^-$.

 For $p>d\ge 1$, let 
\begin{equation}\label{Cpd}
C_{p,d}=\int_{0}^{1}\frac{t^{d-1}\left(1+t^2\right)\left(\log\frac{1}{t}\right)^p}{\left(1-t^2\right)^{d+1}}\,dt.
\end{equation}
This constant already appeared in the work of Parini and Ruf \cite{MR3797987} in the context of a fractional Trudinger--Moser inequality. In particular, Parini and Ruf showed that 
$$
C_{p,d}=\frac{\Gamma(p+1)}{d!}\sum_{k=0}^{\infty}\frac{(d+k-1)!}{k!}\frac{1}{(d+2k)^p}.
$$
For example, 
\begin{align*}
C_{p,1}&=(1-2^{-p})\Ga(p+1)\ze(p),\\
C_{p,2}&=2^{-p-1}\Gamma(p+1)\zeta(p-1),\\
C_{p,3}&=\frac{\Gamma(p+1)}{3\cdot 2^{p+1}}\left[(2^{p-2}-1)\zeta(p-2)-\frac{1}{4}(2^{p}-1)\zeta(p)\right].  
\end{align*}
In general, $C_{p,d}$ may be expressed in terms of the Riemann and Hurwitz zeta functions, see \eqref{Hurwitz}. Moreover, by Lemma \ref{lemma: limiting case p tends to d},
$$
\lim_{p\rightarrow d^+}(p-d)C_{p,d}=2^{-d}.
$$

We now establish two-sided estimates for the weight function $W_{p,d,R}$. These estimates provide an explicit comparison with a simpler logarithmic weight and, in particular, allow us to determine the asymptotic behaviour of $W_{p,d,R}$ as $|x|\to 0^+$.

\begin{prop}\label{prop: asymptotics of WpdR}
Let $W_{p,d,R}$ be given by \eqref{WpdR} and $C_{p,d}$ be given by \eqref{Cpd}. Then, for $0<|x|<R$, 
$$
\frac{2\om_{d-1}}{|x|^d\left(\log\frac{R}{|x|}\right)^p}\left(C_{p,d}-\frac{p\,C_{p+1,d}}{\log\frac{R}{|x|}}\right)<W_{p,d,R}(x)<\frac{2\om_{d-1}C_{p,d}}{|x|^d\left(\log\frac{R}{|x|}\right)^p}.
$$
In consequence, 
$$
W_{p,d,R}(x)\sim \frac{2\om_{d-1}C_{p,d}}{|x|^d\left(\log\frac{R}{|x|}\right)^p},\text{ as }|x|\rightarrow 0^+.
$$
\end{prop}
\begin{proof}
Since $\left(\log\frac{R}{t|x|}\right)^{-p}<\left(\log\frac{R}{|x|}\right)^{-p}$ for $0<t<1$, the upper bound is obvious, hence, we will only prove the lower bound. For $0<r<R$, we have 
\begin{align*}
\int_{0}^{1}\frac{t^{d-1}(1+t^2)}{(1-t^2)^{d+1}}\frac{\left(\log\frac{1}{t}\right)^p}{\left(\log\frac{1}{t}+\log\frac{R}{r}\right)^p}\,dt=\frac{1}{\left(\log\frac{R}{r}\right)^p}\int_{0}^{1}\frac{t^{d-1}(1+t^2)}{(1-t^2)^{d+1}}\frac{\left(\log\frac{1}{t}\right)^p}{\left(1+\frac{\log\frac{1}{t}}{\log\frac{R}{r}}\right)^p}\,dt.    
\end{align*}
Using convexity, we have $(1+a)^{-p}>1-ap$ for $a>0$. Therefore,
\begin{align*}
 \int_{0}^{1}&\frac{t^{d-1}(1+t^2)}{(1-t^2)^{d+1}}\frac{\left(\log\frac{1}{t}\right)^p}{\left(1+\frac{\log\frac{1}{t}}{\log\frac{R}{r}}\right)^p}\,dt\\
 &>\int_{0}^{1}\frac{t^{d-1}(1+t^2)\left(\log\frac{1}{t}\right)^p}{(1-t^2)^{d+1}}\,dt-\frac{p}{\log\frac{R}{r}}\int_{0}^{1}\frac{t^{d-1}(1+t^2)\left(\log\frac{1}{t}\right)^{p+1}}{(1-t^2)^{d+1}}\,dt\\
 &=C_{p,d}-\frac{p\,C_{p+1,d}}{\log\frac{R}{r}}.
\end{align*}
Comparing it with the definition of $W_{p,d,R}$, we obtain the desired estimate. 
\end{proof}

\begin{rem}
The lower bound in Proposition \ref{prop: asymptotics of WpdR} is informative only if $\log\frac{R}{|x|}>\tfrac{pC_{p+1,d}}{C_{p,d}}$, that is $0<|x|<Re^{-\frac{pC_{p+1,d}}{C_{p,d}}}$; otherwise, it is nonpositive.
\end{rem}

The following proposition describes the asymptotic behaviour of
$W_{p,d,R}$, as $|x|\to R^{-}$.
\begin{prop}
We have 
$$
W_{p,d,R}(x)\sim\frac{2^{-d+1}\om_{d-1}\B(d,p-d)}{R^d\left(\log\frac{R}{|x|}\right)^d},\quad\text{ as }|x|\rightarrow R^{-}.
$$
\end{prop}
\begin{proof}
 By \eqref{eq2},
 $$
W_{p,d,R}(x) = \frac{2^{-d+1}\om_{d-1}}{|x|^{d}} \int_{0}^{\infty} \frac{\cosh r}{(\sinh r)^{d+1}} \left( \frac{r}{r+ \log \frac{R}{|x|}} \right)^{p} \, dr.
$$
Let $\ve=\log\frac{R}{|x|}$. Substituting $r\mapsto r\ve$, we have
\begin{align*}
I(\ve):=\int_{0}^{\infty} \frac{\cosh r}{(\sinh r)^{d+1}} \left( \frac{r}{r+ \ve} \right)^{p} \, dr=\ve^{-d}\int_{0}^{\infty}\cosh(r\ve)\left(\frac{r\ve}{\sinh r\ve}\right)^{d+1}\frac{r^{p-d-1}}{(1+r)^p}\,dr.     
\end{align*}
Using the dominated convergence theorem, 
$$
\lim_{\ve\rightarrow 0^+}\int_{0}^{\infty}\cosh(r\ve)\left(\frac{r\ve}{\sinh r\ve}\right)^{d+1}\frac{r^{p-d-1}}{(1+r)^p}\,dr=\int_{0}^{\infty}\frac{r^{p-d-1}}{(1+r)^p}\,dr=\B(d,p-d)
$$
and the assertion follows. 
\end{proof}


\color{black}
\section{Proof of Theorem \ref{thm:critical fractional Hardy half-space} and Asymptotic behaviour of \texorpdfstring{$V_{p,d,R}$}{VpdR}}\label{Section 3}


\smallskip

We begin by establishing an auxiliary result that will be used in the proof of the fractional Hardy inequality in the half-space and in the subsequent analysis of the weight function $V_{p,d,R}$.

\begin{lem}
Let $p>1$, $sp=1$, $R>0$, $\operatorname{supp} u\subset (0,R)$. Let 
$$
V_{p,R}(x)=\frac{2}{x}\int_{0}^{1}\frac{\left(\log\frac{1}{t}\right)^p}{(1-t)^2\left(\log\frac{R}{tx}\right)^p}\,dt.
$$
Then,
\begin{equation}\label{logarithmic Hardy on halfline}
\int_{0}^{\infty}\int_{0}^{\infty}\frac{|u(x)-u(y)|^p}{|x-y|^2}\,dy\,dx\ge\left(\frac{p-1}{p}\right)^p\int_{0}^{R}|u(x)|^p V_{p,R}(x)\,dx.    
\end{equation}
\end{lem}
\begin{proof}
We have 
$$
\int_{0}^{\infty}\int_{0}^{\infty}\frac{|u(x)-u(y)|^p}{|x-y|^2}\,dy\,dx=2\int_{0}^{1}\frac{dy}{(1-y)^2}\int_{0}^{\infty}\frac{|u(x)-u(xy)|^p}{x}\,dx
$$
and the result follows from \eqref{key inequality}.
\end{proof}

We now prove Theorem \ref{thm:critical fractional Hardy half-space} by combining the preceding auxiliary results to establish the fractional Hardy inequality in the half-space.

\begin{proof}[\textbf{Proof of Theorem \ref{thm:critical fractional Hardy half-space}}]
First of all, let us observe that, if $f\colon (-a,\infty)\to\R$, $\supp f\subset(-a,R')$, and $R'+a>0$, then, an easy modification of \eqref{logarithmic Hardy on halfline} yields
\begin{equation}\label{Hardy halfline modified}
 \int_{-a}^{\infty}\int_{-a}^{\infty}\frac{|f(x)-f(y)|^p}{|x-y|^2}\,dy\,dx\ge\left(\frac{p-1}{p}\right)^p\int_{-a}^{R'}|f(x)|^p V_{p,R'+a}(x+a)\,dx  
\end{equation}
Let now $u$ be as in the statement of the theorem and let $\Sd_+:=\Sd\cap\R^d_+$. By the Loss--Sloane formula \cite[(26)]{MR2659764}, we have 
\begin{align*}
\int_{\R^d_+} & \int_{\R^d_+}\frac{|u(x)-u(y)|^p}{|x-y|^{d+1}}\,dy\,dx\\
&=\frac{1}{2}\int_{\Sd}d\om\int_{\{x\colon x\cdot\om=0\}}\mathcal{L}_{\om}(dx)\int_{\{x+t\om\in \mathbb{R}^{d}_{+}\}}\int_{\{x+r\om\in\mathbb{R}^{d}_{+}\}}\frac{|u(x+r\om)-u(x+t\om)|^p}{|r-t|^2}\,dr\,dt\\
&=\int_{\Sd_+}d\om\int_{\{x\colon x\cdot\om=0\}}\mathcal{L}_{\om}(dx)\int_{\{x+t\om\in \mathbb{R}^{d}_{+}\}}\int_{\{x+r\om\in\mathbb{R}^{d}_{+}\}}\frac{|u(x+r\om)-u(x+t\om)|^p}{|r-t|^2}\,dr\,dt,
\end{align*}
where $\mathcal{L}_{\om}$ is the Lebesgue measure on the plane $x\cdot\om=0$. Observe that the set $\{t\in\R\colon x+t\om\in\R^d_+\}$ is always a half-line $\left(-\frac{x_d}{\om_d},\infty\right)$ and, for fixed $x$ and $\om$, the function $f(r)=u(x+r\om)$ has support contained in $\left(-\tfrac{x_d}{w_d},\frac{R-x_d}{\om_d}\right)$. Therefore, using \eqref{Hardy halfline modified}, 
\begin{align*}
\int_{\{x+t\om\in\R^d_+\}} & \int_{\{x+r\om\in\R^d_+\}}\frac{|u(x+r\om)-u(x+t\om)|^p}{|r-t|^2}\,dr\,dt\\
&\ge \left(\frac{p-1}{p}\right)^p\int_{\{x+r\om\in A_R\}}|u(x+r\om)|^p V_{p,\frac{R}{\om_d}}\left(r+\tfrac{x_d}{\om_d}\right)\,dr\\
&=2\left(\frac{p-1}{p}\right)^p\int_{\{x+r\om\in A_R\}}\int_{0}^{1}\frac{|u(x+r\om)|^p}{r+\tfrac{x_d}{\om_d}}\frac{\left(\log\frac{1}{t}\right)^p}{(1-t)^2\left(\log\frac{R/\om_d}{t(r+x_d/\om_d)}\right)^p}\,dt\,dr\\
&=2\left(\frac{p-1}{p}\right)^p\om_d\int_{\{x+r\om\in A_R\}}\int_{0}^{1}\frac{|u(x+r\om)|^p}{r\om_d+x_d}\frac{\left(\log\frac{1}{t}\right)^p}{(1-t)^2\left(\log\frac{R}{t(r\om_d+x_d)}\right)^p}\,dt\,dr.
\end{align*}

Hence, 
\begin{align*}
\int_{\R^d_+}\int_{\R^d_+}&\frac{|u(x)-u(y)|^p}{|x-y|^{d+1}}\,dy\,dx\\
&\ge2\left(\frac{p-1}{p}\right)^p\int_{\Sd_+}\om_d\,d\om\int_{\{x\colon x\cdot\om=0\}}\mathcal{L}_{\om}(dx)\\
&\quad\times\int_{\{x+r\om\in A_R\}}\int_{0}^{1}\frac{|u(x+r\om)|^p}{r\om_d+x_d}\frac{\left(\log\frac{1}{t}\right)^p}{(1-t)^2\left(\log\frac{R}{t(r\om_d+x_d)}\right)^p}\,dt\,dr\\
&=\left(\frac{p-1}{p}\right)^p\int_{A_R}|u(x)|^p V_{p,d,R}(x_d)\,dx,
\end{align*}
where the last equality follows from the well known fact
$$
\int_{\Sd}|\om_d|^{\al}\,d\om=\frac{2\pi^{\frac{d-1}{2}}\Gamma\left(\frac{1+\al}{2}\right)}{\Gamma\left(\frac{d+\al}{2}\right)},\quad\al>-1.
$$
It remains to justify sharpness of the constant $\left(\tfrac{p-1}{p}\right)^p$ in \eqref{halfspace}. To this end, assume first $d=1$. Then, the sharpness can be obtained by a similar way as in the proof of Theorem \ref{Thm: logarithmic frac Hardy inequality}, with minor modifications, by considering test functions of the form
$$
u_N(x)=\begin{cases}
   0,&\quad0<x\le Re^{-2N},\,\tfrac{R}{e}\le x<R,\\
   \left(\log\frac{R}{x}\right)^{\frac{p-1}{p}}-1,&\quad Re^{-N}\le x\le \frac{R}{e},\\
   \left(N^{\frac{p-1}{p}}-1\right)\left(2-\frac{\log\frac{R}{x}}{N}\right),&\quad Re^{-2N}\le x\le Re^{-N}.
\end{cases}
$$
 Let now $d\ge 2$. We borrow the idea from \cite[Proof of Theorem 6.2]{MR4800921} and reduce the problem to the one-dimensional situation. Let us fix $\chi\in C^{1}_c(\R^{d-1})$ and $\eta\in C_c((0,R))$. For $N>0$, we define
$$
u_N(x)=\chi_N(x')\eta(x_d),
$$
where $\chi_N(x')=\frac{N^{\frac{1-d}{p}}}{\|\chi\|_{L^p(\R^{d-1})}}\chi\left(\frac{x'}{N}\right)$. Then, of course, $\supp u_N\subset A_R=\{0<x_d<R\}$ and 
$$
\int_{A_R}|u_N(x)|^pV_{p,d,R}(x_d)\,dx=\frac{\kappa_d}{2}\int_{0}^R|\eta(r)|^p V_{p,R}(r)\,dr,
$$
where $V_{p,R}=V_{p,1,R}$. Moreover, using triangle inequality, the same calculations as in \cite[Proof of Theorem 6.2]{MR4800921} yield
\begin{align*}
&\frac{[u_N]^p_{W^{1/p,p}(\R^d_+)}}{\displaystyle\int_{A_R}|u_N(x)|^pV_{p,d,R}(x_d)\,dx}\\
&\le\left[\frac{[\eta]_{W^{1/p,p}((0,\infty))}}{\displaystyle\left(\int_{0}^R|\eta(r)|^p V_{p,R}(r)\,dr\right)^\frac{1}{p}}+C\frac{\|\eta\|_{L^p((0,\infty))}}{N^{\frac{1}{p}}\|\chi\|_{L^p(\R^{d-1})}}\frac{[\chi]_{W^{1/p,p}(\R^{d-1})}}{\left(\displaystyle\int_{0}^{R}|\eta(r)|^p V_{p,R}(r)\,dr\right)^\frac{1}{p}}\right]^p,    
\end{align*}
for some $C>0$. We now let $N\rightarrow\infty$. Therefore, since $V_{p,d,R}=\frac{\kappa_d}{2}V_{p,R}$, sharpness for $d\ge 2$ follows from sharpness for $d=1$.
\end{proof}

\subsection{Asymptotic behaviour of \texorpdfstring{$V_{p,d,R}$}{VpdR}} In this subsection, we investigate the asymptotic behaviour of the weight function $V_{p,d,R}$ associated with the fractional Hardy inequality in the half-space, as $x_d\to 0^+$ and $x_d\to R^-$.

\smallskip

We first establish two-sided bounds for the weight function $V_{p,d,R}$ and determine its asymptotic behaviour as $x_d\to 0^+$.

\begin{prop}
Let $V_{p,d,R}$ be given by \eqref{VpdR}, $p>1$, $d\ge 1$ $R>0$. Then, for all $0<x_d<R$,
\begin{align*}
\frac{\kappa_d}{x_d\left(\log\frac{R}{x_d}\right)^p}&\left(\Ga(p+1)\ze(p)-\frac{p\,\Ga(p+2)\ze(p+1)}{\log\frac{R}{x_d}}\right)<V_{p,d,R}(x_d)<\frac{\kappa_d\Ga(p+1)\ze(p)}{x_d\left(\log\frac{R}{x_d}\right)^p}.
\end{align*}
In particular, 
$$
V_{p,d,R}(x_d)\sim \frac{\kappa_d\Ga(p+1)\ze(p)}{x_d\left(\log\frac{R}{x_d}\right)^p},\quad\text{ as }x_d\rightarrow 0^+.
$$
\end{prop}
\begin{proof}
Consider 
$$
C_p=\int_{0}^{1}\frac{\left(\log\frac{1}{t}\right)^p}{(1-t)^2}\,dt.
$$
Expanding $(1-t)^{-2}$ into Taylor series, it may be easily shown that $C_p=\Ga(p+1)\ze(p)$. Hence, the proof of the asymptotics follow the same pattern as the proof of Proposition \ref{prop: asymptotics of WpdR} and we omit the details.
\end{proof}

Of course, the lower bound in the estimate above is nontrivial for $0<x_d<R e^{-\frac{p(p+1)\ze(p+1)}{\ze(p)}}$.

\smallskip

The next proposition describes the asymptotic behaviour of the weight function $V_{p,d,R}$ as $x_d\to R^-$.

\begin{prop}
It holds
$$
V_{p,d,R}(x_d)\sim\frac{\kappa_d}{(p-1)}\frac{1}{R\log\frac{R}{x_d}},\quad\text{ as }x_d\rightarrow R^-.
$$
\end{prop}
\begin{proof}
Let $\ve=\log\frac{R}{x_d}$. Substituting $t\mapsto e^{-r}$, we have 
\begin{align*}
\int_{0}^{1}\frac{\left(\log\frac{1}{t}\right)^p}{(1-t)^2\left(\log\frac{1}{t}+\ve\right)^p}\,dt&=\int_{0}^{\infty}\frac{r^pe^{-r}}{\left(1-e^{-r}\right)^2(\ve+r)^p}\,dr\\
&=\frac{1}{\ve}\int_{0}^{\infty}\left(\frac{r\ve}{1-e^{-r\ve}}\right)^2\frac{r^{p-2}}{(1+r)^{p}}e^{-r\ve}\,dr
\end{align*}
and, by dominated convergence, 
\begin{align*}
\lim_{\ve\rightarrow 0^+}\int_{0}^{\infty}\left(\frac{r\ve}{1-e^{-r\ve}}\right)^2\frac{r^{p-2}}{(1+r)^{p}}e^{-r\ve}\,dr=\int_{0}^{\infty}\frac{r^{p-2}}{(1+r)^p}\,dr=\B(1,p-1)=\frac{1}{p-1}.   
\end{align*}
The proof is complete.
\end{proof}
\color{black}

\section{Appendix}\label{Section 4}

In this section, we establish a Bourgain--Brezis--Mironescu type formula in the case $p>d$, where $d \geq 1$. In particular, we prove the asymptotic formula as $p \to d$, which yields the corresponding local integral involving the gradient of the function.

\begin{thm}\label{Thm: BBM type Formula}
Let $d\geq 1$, $u\in C_c^1(\mathbb R^d)$ with
$\operatorname{supp} u\subset B_R$. Then
\begin{equation*}
\lim_{p \to d^+}(p-d)
\int_{\mathbb R^d}\int_{\mathbb R^d}
\frac{|u(x)-u(y)|^{p}}{|x-y|^{2d}}\,dy\,dx =
\frac{2\pi^{\frac{d-1}{2}}
\Gamma\left(\frac{1+d}{2}\right)}
{\Gamma(d)}
\int_{B_R}|\nabla u(x)|^d\,dx.
\end{equation*}
\end{thm}

\begin{proof}
Making the change of
variables $y=x+h$, we obtain
\begin{equation}\label{eq:change}
(p-d)\int_{\R^d}\int_{\R^d}
\frac{|u(x)-u(y)|^p}{|x-y|^{2d}}\,dy\,dx =(p-d)\int_{\R^d}
\frac{1}{|h|^{2d}}
\int_{\R^d}|u(x+h)-u(x)|^p\,dx\,dh.
\end{equation}
Fix $\delta\in(0,1)$. We split the integral on the
right-hand side of \eqref{eq:change} into the regions
$|h|\ge\delta$ and $|h|<\delta$.

\smallskip

\noindent For the first integral, using $|u(x+h)-u(x)|^p
\leq 2^{p-1}\bigl(|u(x+h)|^p+|u(x)|^p\bigr)$, we obtain
\begin{equation*}
\int_{\R^d}|u(x+h)-u(x)|^p\,dx \leq 2^p\int_{\R^d}|u(x)|^p\,dx \leq 2^p\|u\|_{L^{\infty}(\mathbb{R}^{d})}^p|B_R|.
\end{equation*}
We choose $p$ such that
$d<p\leq d+1$. Thus the last expression is bounded
independently of $p$. Consequently, there exists a constant
$C>0$, independent of $p$ and $\delta$, such that
\begin{align*}
  0 & \leq (p-d)\int_{|h|\geq\delta}
\frac{1}{|h|^{2d}}
\int_{\R^d}|u(x+h)-u(x)|^p\,dx\,dh \\
& \leq C(p-d)\int_{|h|\geq\delta}|h|^{-2d}\,dh = C(p-d)\om_{d-1}
\int_\delta^\infty r^{-d-1}\,dr=\frac{C\om_{d-1}}{d}
(p-d)\delta^{-d}.  
\end{align*}
Hence, for every fixed $\delta>0$,
\begin{equation}\label{eq:far}
\lim_{p \to d^+}(p-d)\int_{|h|\geq\delta}
\frac{1}{|h|^{2d}}
\int_{\R^d}|u(x+h)-u(x)|^p\,dx\,dh=0.
\end{equation}

\smallskip

\noindent Let us now focus on the second integral. Since $u\in C_c^1(\mathbb R^d)$, its gradient is uniformly
continuous. Define
\begin{equation*}
   \omega(\delta) :=
\sup_{\substack{a,b\in\R^d \\ |a-b|\leq\delta}}
|\nabla u(a)-\nabla u(b)|. 
\end{equation*}
Then $\omega(\delta)\to0$ as $\delta\to0^+$. Fix $\sigma\in\mathbb S^{d-1}$ and $0<r<\delta<1$. By the fundamental theorem of calculus,
\begin{equation*}
    u(x+r\sigma)-u(x)
=r\int_0^1\nabla u(x+tr\sigma)\cdot\sigma\,dt.
\end{equation*}
Therefore,
\begin{equation*}
    u(x+r\sigma)-u(x)-r\nabla u(x)\cdot\sigma =r\int_0^1
\bigl(\nabla u(x+tr\sigma)-\nabla u(x)\bigr)
\cdot\sigma\,dt.
\end{equation*}
It follows that
\begin{equation*}
    |u(x+r\sigma)-u(x)-r\nabla u(x)\cdot\sigma|
\leq r\omega(\delta).
\end{equation*}
Both $u(x+r\sigma)-u(x)$ and
$r\nabla u(x)\cdot\sigma$ vanish outside $B_{R+1}$. Consequently,
\begin{equation*}
\left(
\int_{\R^d}
|u(x+r\sigma)-u(x)-r\nabla u(x)\cdot\sigma|^p
\,dx\right)^{1/p} \leq r\omega(\delta)|B_{R+1}|^{1/p} \leq Cr\omega(\delta), 
\end{equation*}
where $C=\max\{1,|B_{R+1}|\}$ is independent of
$p$, $r$, and $\sigma$. By Minkowski's inequality and its reverse form, we obtain
\begin{align*}
r\left(
\left(\int_{\R^d}|\nabla u(x)\cdot\sigma|^p\,dx
\right)^{1/p}
-C\omega(\delta)\right) & \leq
\left(\int_{\R^d}|u(x+r\sigma)-u(x)|^p\,dx
\right)^{1/p} \\ & \quad \leq
r\left(
\left(\int_{\R^d}|\nabla u(x)\cdot\sigma|^p\,dx
\right)^{1/p}
+C\omega(\delta)\right).  
\end{align*}
Since the middle expression is nonnegative, we may
replace the lower bound by its positive part. Thus
\begin{align}\label{ineq09}
    r^p\left(
\left(\int_{\R^d}|\nabla u(x)\cdot\sigma|^p\,dx
\right)^{1/p}
-C\omega(\delta)\right)_+^p \nonumber & \leq
\int_{\R^d}|u(x+r\sigma)-u(x)|^p\,dx \\ & \quad \leq
r^p\left(
\left(\int_{\R^d}|\nabla u(x)\cdot\sigma|^p\,dx
\right)^{1/p}
+C\omega(\delta)\right)^p,
\end{align}
where $a_+=\max\{a,0\}$.

Using polar coordinates, we have
\begin{align*}
    (p-d)\int_{|h|<\delta}
\frac{1}{|h|^{2d}}
\int_{\R^d} &|u(x+h)-u(x)|^p\,dx\,dh \\ & =(p-d)\int_{\mathbb S^{d-1}}\int_0^\delta
r^{-d-1}
\int_{\R^d}|u(x+r\sigma)-u(x)|^p\,dx
\,dr\,d\sigma.
\end{align*}
Also, we have
\begin{equation*}
 (p-d)\int_0^\delta r^{p-d-1}\,dr
=\delta^{p-d}.
\end{equation*}
Applying the inequalities obtained in \eqref{ineq09}, we get
\begin{align*}
 \delta^{p-d}
\int_{\mathbb S^{d-1}} &
\left(
\left(\int_{\R^d}|\nabla u(x)\cdot\sigma|^p\,dx
\right)^{1/p}
-C\omega(\delta)\right)_+^p\,d\sigma\\
&\quad\leq
(p-d)\int_{|h|<\delta}
\frac{1}{|h|^{2d}}
\int_{\R^d}|u(x+h)-u(x)|^p\,dx\,dh\\
&\quad\leq
\delta^{p-d}
\int_{\mathbb S^{d-1}}
\left(
\left(\int_{\R^d}|\nabla u(x)\cdot\sigma|^p\,dx
\right)^{1/p}
+C\omega(\delta)\right)^p\,d\sigma.  
\end{align*}
As $p \to d$, we have $\delta^{p-d}\to1$. Since $\nabla u$ is bounded and supported in $B_R$,
\begin{equation*}
    \int_{\R^d}|\nabla u(x)\cdot\sigma|^p\,dx
\longrightarrow
\int_{\R^d}|\nabla u(x)\cdot\sigma|^d\,dx
\end{equation*}
uniformly in $\sigma\in\mathbb S^{d-1}$. Taking the limit inferior and limit superior in the preceding inequalities, and using \eqref{eq:far}, we obtain
\begin{align*}
  \int_{\mathbb S^{d-1}} &
\left(
\left(\int_{\R^d}|\nabla u(x)\cdot\sigma|^d\,dx
\right)^{1/d}
-C\omega(\delta)\right)_+^d\,d\sigma \\ & \quad \leq
\liminf_{p \to d^+}(p-d)
\int_{\R^d}\int_{\R^d}
\frac{|u(x)-u(y)|^p}{|x-y|^{2d}}\,dy\,dx\\
&\quad\leq
\limsup_{p \to d^+}(p-d)
\int_{\R^d}\int_{\R^d}
\frac{|u(x)-u(y)|^p}{|x-y|^{2d}}\,dy\,dx\\
&\quad\leq
\int_{\mathbb S^{d-1}}
\left(
\left(\int_{\R^d}|\nabla u(x)\cdot\sigma|^d\,dx
\right)^{1/d}
+C\omega(\delta)\right)^d\,d\sigma.  
\end{align*}
Letting $\delta\to0^+$ and using
$\omega(\delta)\to0$, the dominated convergence theorem
gives
\begin{equation*}
  \lim_{p \to d^+}(p-d)
\int_{\R^d}\int_{\R^d}
\frac{|u(x)-u(y)|^{p}}{|x-y|^{2d}}\,dy\,dx =\int_{\mathbb S^{d-1}}
\int_{\R^d}|\nabla u(x)\cdot\sigma|^d\,dx\,d\sigma.
\end{equation*}
By rotation invariance of surface measure,
\begin{equation*}
    \int_{\mathbb S^{d-1}}|\nabla u(x)\cdot\sigma|^d\,d\sigma
=|\nabla u(x)|^d
\int_{\mathbb S^{d-1}}|\sigma_1|^d\,d\sigma.
\end{equation*}
Therefore, by Tonelli's theorem,
\begin{equation*}
\lim_{p \to d^+}(p-d)
\int_{\R^d}\int_{\R^d}
\frac{|u(x)-u(y)|^p}{|x-y|^{2d}}\,dy\,dx=
\left(\int_{\mathbb S^{d-1}}|\sigma_1|^d\,d\sigma\right)
\int_{\R^d}|\nabla u(x)|^d\,dx.
\end{equation*}
Also, we have
\begin{equation*}
    \int_{\mathbb S^{d-1}}|\sigma_1|^d\,d\sigma
=\frac{2\pi^{\frac{d-1}{2}}
\Gamma\left(\frac{d+1}{2}\right)}{\Gamma(d)}.
\end{equation*}
Substituting this identity into the limiting formula gives the assertion of the theorem.
\end{proof}

\bigskip

\textbf{Acknowledgement:} V. Sahu is supported by the Anusandhan National Research Foundation (ANRF) through the National Postdoctoral Fellowship (PDF/2025/004611).

\end{document}